\documentclass[11pt]{article}%
\usepackage{amsfonts}
\usepackage{amsmath}
\usepackage{amssymb}
\usepackage{graphicx}%
\providecommand{\U}[1]{\protect\rule{.1in}{.1in}}
\newtheorem{theorem}{Theorem}

\newtheorem{definition}[theorem]{Definition}

\newtheorem{lemma}[theorem]{Lemma}
\newtheorem{notation}[theorem]{Notation}

\newtheorem{proposition}[theorem]{Proposition}
\newtheorem{remark}[theorem]{Remark}

\newenvironment{proof}[1][Proof]{\noindent\textbf{#1.} }{\ \rule{0.5em}{0.5em}}
{\catcode`\@=11\global\let\AddToReset=\@addtoreset
\AddToReset{equation}{section}

\AddToReset{theorem}{section}

\begin{document}

\title{Backward blow-up estimates and initial trace for a parabolic system of reaction-diffusion}
\author{Marie-Fran\c{c}oise BIDAUT-VERON\thanks{Laboratoire de Math\'{e}matiques et
Physique Th\'{e}orique, CNRS UMR 6083, Facult\'{e} des SCiences, 37200 Tours
France. E-mail address: veronmf@univ-tours.fr}
\and Marta GARC\'{I}A-HUIDOBRO\thanks{Departamento de Matem\'{a}ticas, Pontificia
Universidad Cat\'{o}lica de Chile, Casilla 306, Correo 22, Santiago de Chile.
E-mail address: mgarcia@mat.puc.cl}
\and Cecilia YARUR\thanks{Departamento de Matem\'{a}tica y C.C., Universidad de
Santiago de Chile, Casilla 307, Correo 2, Santiago de Chile. E-mail address:
cecilia.yarur@usach.cl}}
\date{.}
\maketitle

\begin{abstract}
In this article we study the positive solutions of the parabolic semilinear
system of competitive type
\[
\left\{
\begin{array}
[c]{c}%
u_{t}-\Delta u+v^{p}=0,\\
v_{t}-\Delta v+u^{q}=0,
\end{array}
\right.
\]
in $\Omega\times\left(  0,T\right)  $, where $\Omega$ is a domain of
$\mathbb{R}^{N},$ and $p,q>0,$ $pq\neq1.$ Despite of the lack of comparison
principles, we prove local upper estimates in the superlinear case $pq>1$ of
the form
\[
u(x,t)\leqq Ct^{-(p+1)/(pq-1)},\qquad v(x,t)\leqq Ct^{-(q+1)/(pq-1)}%
\]
in $\omega\times\left(  0,T_{1}\right)  ,$ for any domain $\omega
\subset\subset\Omega$ and $T_{1}\in\left(  0,T\right)  ,$ and $C=C(N,p,q,T_{1}%
,\omega).$ For $p,q>1,$ we prove the existence of an initial trace at time 0,
which is a Borel measure on $\Omega.$ Finally we prove that the punctual
singularities at time $0$ are removable when $p,q\geqq1+2/N.\bigskip\bigskip$

\bigskip

\noindent\textbf{A.M.S. Classification: 35K45, 35K57, 35K58.}

\bigskip

\noindent This research was supported by Fondecyt 7090028 for the first
author, by Fondecyt 1070125 and 1070951 for the second and third author and
ECOS-CONICYT C08E04 for the three authors.

\end{abstract}

\pagebreak

\section{Introduction \label{intro}}

\noindent Let $\Omega$ be a domain of $\mathbb{R}^{N}$ ($N\geq1$) and
$0<T\leqq\infty.$ In this work we are concerned with the positive solutions of
the parabolic system with absorption terms
\begin{equation}
\left\{
\begin{array}
[c]{c}%
u_{t}-\Delta u+v^{p}=0,\\
v_{t}-\Delta v+u^{q}=0,
\end{array}
\right.  \label{one}%
\end{equation}
in $\Omega\times\left(  0,T\right)  ,$ with $p,q>0,$ $pq\neq1$, in particular
in the superlinear case where $pq>1.$ \medskip\ 

This system appears as a simple model of competition between two species,
where the increase of the population of one of them reduces the growth rate of
the other. Independently of the biological applications, it presents an
evident interest, since it is the direct extension of the scalar equation
\begin{equation}
U_{t}-\Delta U+U^{Q}=0, \label{sca}%
\end{equation}
with $Q\neq1.$ For $Q>1,$ any nonnegative subsolution of equation (\ref{sca})
in $\Omega\times\left(  0,T\right)  $ satisfies the following upper estimate:
for any bounded $C^{2}$ domain $\omega\subset\Omega$
\begin{equation}
U(x,t)\leqq((Q-1)t)^{-1/(Q-1)}+Cd(x,\partial\omega)^{-2/(Q-1)}\qquad
\forall(x,t)\in\omega\times\left(  0,T\right)  , \label{mai}%
\end{equation}
where $d(x,\partial\omega)$ is the distance from $x$ to the boundary of
$\omega$ and $C=C(N,Q),$ see \cite{MV}. This estimate follows from the
comparison principle, as shown at Proposition \ref{set}. Moreover it was
proved in \cite{MV} that any solution $U$ of equation (\ref{sca}) in
$\Omega\times\left(  0,T\right)  $ admits a trace at time $0$ in the following
sense:\medskip

There exist two disjoints sets $\mathcal{R}$ and $\mathcal{S}$ such that
$\mathcal{R\cup S}=\Omega$, and $\mathcal{R}$ is open, and a nonnegative Radon
measure $\mu$ on $\mathcal{R}$, such that\medskip

$\bullet$ For any $x_{0}\in\mathcal{R}$, and any $\psi\in C_{c}^{0}%
(\mathcal{R)}$,
\[
\lim_{t\rightarrow0}\int_{\mathcal{R}}U(.,t)\psi=\int_{\mathcal{R}}\psi d\mu,
\]
\medskip

$\bullet$ For any open set $\mathcal{U}$ such that $\mathcal{U}\cap
\mathcal{S\neq\emptyset}$,%
\[
\lim_{t\rightarrow0}\int_{\mathcal{U}}u(.,t)=\infty.
\]
Moreover the trace $(\mathcal{S},\mu)$ is unique whenever $Q<1+2/N.\medskip$

Up to now, system (\ref{one}) has been barely touched on. Indeed an essential
difficulty appears: \textit{the lack of results for comparison principles. }As
a consequence, most of the classical properties of equation (\ref{sca}), based
on the use of standard supersolutions, are hardly extendable. Some existence
results are given in \cite{Ka} for bounded initial data, and then in
\cite{BiGaYa} for more general multipower systems with non smooth data, see
also \cite{LeYo} for quasilinear operators. Otherwise the existence of
traveling waves is treated in \cite{EsH}. For the associated elliptic system
\begin{equation}
\left\{
\begin{array}
[c]{c}%
-\Delta u+v^{p}=0,\\
-\Delta v+u^{q}=0,
\end{array}
\right.  \label{ell}%
\end{equation}
the isolated singularities are completely described in \cite{BiGr} for the
superlinear case $pq>1$ and for the sublinear case $pq<1$, see also \cite{Ya},
\cite{Ya2} for $p,q\geqq1.$ The study shows a great complexity of the possible
singularities; in particular many nonradial singular solutions are constructed
by bifurcation methods. The question of large solutions of system (\ref{ell})
is studied in the radial case in \cite{GaSaLe}, showing unexpected
multiplicity results, and the behavior of the solutions near the boundary is
open in dimension $N>1;$ the existence is an open problem in the general case.
For such competitive problems, some more adapted sub-supersolutions and
super-subsolutions have been introduced, see \cite{MaMi}, \cite{Pa},
\cite{BiGaYa}, \cite{GaRo}, but the problem remains to construct them. The
uniqueness is also a difficult problem, as it was first observed in
\cite{Ar}.\medskip

Our first result consists in local backward upper estimates for the solutions
of the system: defining the two exponents
\begin{equation}
a=\frac{p+1}{pq-1},\qquad b=\frac{q+1}{pq-1}, \label{ab}%
\end{equation}
we obtain the following:\medskip\ 

\begin{theorem}
\label{intest} Assume that $pq>1.$ Let $(u,v)$ be a positive solution of
system (\ref{one}) in $\Omega\times\left(  0,T\right)  .$ Then for any domain
$\omega\subset\subset\Omega$ $(\omega=\mathbb{R}^{N}$ if $\Omega
=\mathbb{R}^{N}),$
\begin{equation}
u(x,t)\leqq Ct^{-a},\qquad v(x,t)\leqq Ct^{-b},\qquad\forall(x,t)\in
\omega\times\left(  0,T\right)  , \label{upe}%
\end{equation}
for some $C=C(N,p,q,T,\omega).$\medskip\ 
\end{theorem}

Our second result is the existence of a trace in the following sense:\medskip\ 

\begin{theorem}
\label{trace} ~Assume that $p,q>1.$ Let $(u,v)$ be a positive solution of the
system in $\Omega\times\left(  0,T\right)  .$ Then there exist two disjoints
sets $\mathcal{R}$ and $\mathcal{S}$ such that $\mathcal{R\cup S}=\Omega$, and
$\mathcal{R}$ is open, and nonnegative Radon measures $\mu_{1},\mu_{2}$ on
$\mathcal{R}$, such that the following holds: \medskip

$\bullet$ For any $x_{0}\in\mathcal{R}$, and any $\psi\in C_{c}^{0}%
(\mathcal{R)},$
\begin{equation}
\lim_{t\rightarrow0}\int_{\mathcal{R}}u(.,t)\psi=\int_{\mathcal{R}}\psi
d\mu_{1},\qquad\lim_{t\rightarrow0}\int_{\mathcal{R}}v(.,t)\psi=\int
_{\mathcal{R}}\psi d\mu_{2}. \label{ccc}%
\end{equation}
\medskip

$\bullet$ For any open set $\mathcal{U}$ such that $\mathcal{U}\cap
\mathcal{S\neq\emptyset},$%
\begin{equation}
\lim_{t\rightarrow0}\int_{\mathcal{U}}(u(.,t)+v(.,t))=\infty. \label{ddd}%
\end{equation}
\medskip\ 
\end{theorem}

As a consequence we can define a notion of trace of $(u,v)$ at time
$0$:\medskip\ 

\begin{definition}
The couple $\mathcal{B}=(\mathcal{B}_{1},\mathcal{B}_{2})$ of Borel measures
$\mathcal{B}_{1}$,$\mathcal{B}_{2}$ on $\Omega$ associated to the triplet
$(\mathcal{S},\mu_{1},\mu_{2})$ defined for $i=1$,$2$ by
\[
\mathcal{B}_{i}(E)=\left\{
\begin{array}
[c]{c}%
\mu_{i}(E)\qquad\text{if }E\subset\mathcal{R},\\
\infty\qquad\quad\text{if }E\cap\mathcal{S\neq\emptyset},
\end{array}
\right.
\]
is called the initial trace of $(u,v).\medskip$
\end{definition}

Finally we give a result of removability of the initial singularities inspired
by \cite[Theorem 2]{BrFr}:\medskip\ 

\begin{theorem}
\label{rem} Assume that $p,q\geqq1+2/N.$ If there exists a positive solution
$(u,v)$ of system (\ref{one}) in $\Omega\times\left(  0,T\right)  $ such that
\begin{equation}
\lim_{t\rightarrow0}\int_{\Omega}(u(.,t)+v(t))\varphi=0,\qquad\forall
\varphi\in C_{c}^{\infty}\left(  \Omega\backslash\left\{  0\right\}  \right)
, \label{dil}%
\end{equation}
then $u,v\in C^{2,1}(\Omega\times\left[  0,T\right)  )$ and $u(x,0)=v(x,0)=0,$
$\forall x\in\Omega.$\bigskip
\end{theorem}

In each section we point out some questions which remain open.

\section{Some existence results\label{exi}}

Next we recall some results that we obtained in \cite{BiGaYa} where we studied
the existence and the eventual uniqueness of signed solutions of the Cauchy
problem with initial data $(u_{0},v_{0})$
\begin{equation}
\left\{
\begin{array}
[c]{c}%
u_{t}-\Delta u+\left\vert v\right\vert ^{p}\left\vert u\right\vert ^{-1}u=0,\\
v_{t}-\Delta v+\left\vert u\right\vert ^{q}\left\vert v\right\vert ^{-1}v=0,
\end{array}
\right.  \label{pdo}%
\end{equation}
where $p,q>0$, and
\[
\left\vert u\right\vert ^{-1}u=\left\{
\begin{array}
[c]{c}%
1\quad\text{ if }u>0,\\
0\text{ \quad if }u=0,\\
-1\text{ \quad if }u<0.
\end{array}
\right.
\]
In particular we showed in \cite{BiGaYa} the following results:\medskip

\begin{theorem}
\label{meas} Assume that $\Omega$ is bounded. Suppose that $u_{0}\in
L^{\theta}(\Omega)$ and $v_{0}\in L^{\lambda}(\Omega)$ with $1\leqq
\theta,\lambda\leqq\infty$, with
\[
\max(\frac{p}{\lambda},\frac{q}{\theta})<1+2/N,
\]
or that $u_{0}$,$v_{0}$ are two bounded Radon measures in $\Omega$, and
\begin{equation}
\max(p,q)<1+2/N. \label{sho}%
\end{equation}
Then there exists a weak solution ($u,v)$ of the system with Dirichlet or
Neuman conditions on the lateral boundary, such that for any $\psi\in
C_{c}^{0}(\Omega)$,
\begin{equation}
\lim_{t\rightarrow0}\int_{\mathcal{R}}u(.,t)\psi=\int_{\mathcal{R}}\psi
du_{0},\qquad\lim_{t\rightarrow0}\int_{\mathcal{R}}v(.,t)\psi=\int
_{\mathcal{R}}\psi dv_{0}. \label{gen}%
\end{equation}
Also, there exist two solutions $(u_{1},v_{1})$ and $(u_{2},v_{2})$ such that
any solution $(u,v)$ satisfies $u_{1}\leqq u\leqq u_{2}$ and $v_{2}\leqq
v\leqq v_{1}.\bigskip$

Moreover, if $p,q\geq1$ and $u_{0}\in L^{\theta}(\Omega)$ and $v_{0}\in
L^{\lambda}(\Omega)$ with%
\begin{equation}
\max(\frac{p}{\lambda}-\frac{1}{\theta},\frac{q}{\theta}-\frac{1}{\lambda
})<\frac{2}{N}, \label{xyz}%
\end{equation}
then $(u,v)$ is unique; in particular this holds for any $u_{0},v_{0}\in
L^{1}(\Omega),$ if (\ref{sho}) is satisfied, or if $u_{0},v_{0}\in L^{\theta
}(\Omega)$ with $\theta\geq N(\max(p,q)-1)/2.$
\end{theorem}

\section{Local a priori estimates}

When looking for local upper estimates of the nonnegative solutions of system
(\ref{pdo}) near $t=0$, we notice that the system admits the solution $(0,v)$
with $v$ a solution of the heat equation in $\Omega\times\left(  0,T\right)
$, for which we have no estimate, since the set of solutions is a vector
space. That is why we suppose that $u$ and $v$ are \textit{positive} in
$\Omega\times\left(  0,T\right)  .$ The question of upper estimates for one of
the functions is very closely linked to the question of lower estimates for
the other one.\medskip

We define a solution of problem (\ref{one}) in $\Omega\times\left(
0,T\right)  $ as a couple $(u,v)$ of positive functions such that $u\in
L_{loc}^{q}(\Omega\times\left(  0,T\right)  )$, $v\in L_{loc}^{p}(\Omega
\times\left(  0,T\right)  )$ and
\begin{align}%
{\displaystyle\iint_{\Omega\times\left(  0,T\right)  }}
\left(  -u\varphi_{t}-u\Delta\varphi+v^{p}\varphi\right)   &  =0,\label{wea}\\%
{\displaystyle\iint_{\Omega\times\left(  0,T\right)  }}
\left(  -v\varphi_{t}-v\Delta\varphi+u^{q}\varphi\right)   &  =0, \label{wi}%
\end{align}
for any $\varphi\in\mathcal{D}(\Omega\times\left(  0,T\right)  ).$ From the
standard regularity theory for the heat equation it follows that $u,v\in
C_{loc}^{2,1}(\Omega\times\left(  0,T\right)  )$, and then $u,v\in C^{\infty
}(\Omega\times\left(  0,T\right)  )$ since $u,v$ are positive.\medskip

As in the case of the scalar equation (\ref{sca}), the system (\ref{one})
admits a particular solution $(u^{\ast},v^{\ast})$ for $pq>1$, defined by
\[
u^{\ast}(t)=A^{\ast}t^{-a},\qquad v^{\ast}(t)=B^{\ast}t^{-b},
\]
where
\[
(A^{\ast})^{pq-1}=(p+1)(q+1)^{p}(pq-1)^{-(p+1)},\qquad(B^{\ast})^{pq-1}%
=(q+1)(p+1)^{q}(pq-1)^{-(q+1)}.
\]
\medskip

In \cite{BiGr}, the authors studied the singularities near $0$ of the positive
solutions of the associated elliptic system (\ref{ell}) in $B(0,1)\backslash
\left\{  0\right\}  $. System (\ref{ell}) admits particular solutions when
$\min(2a,2b)>N-2$, given by
\[
u_{\ast}(x)=A_{\ast}\left\vert x\right\vert ^{-2a},\qquad v_{\ast}(x)=B_{\ast
}\left\vert x\right\vert ^{-2b},
\]
with
\[
A_{\ast}^{pq-1}=2a(2a+2-N)((2b(2b+2-N))^{p},\qquad B_{\ast}^{pq-1}%
=2b(2b+2-N)((2a(2a+2-N))^{p}.
\]
When $pq>1$ the following upper estimates hold near near $0:$%
\[
u(x)\leqq C\left\vert x\right\vert ^{-2a},\qquad v(x)\leqq C\left\vert
x\right\vert ^{-2b},
\]
for some $C=C(p,q,N)$. The proofs were based on estimates of the mean value of
$u$ and $v$ on the sphere $\left\{  \left\vert x\right\vert =r\right\}  $, on
the mean value inequality for subharmonic functions, and a bootstrap technique
for comparisons between different spheres.\medskip

For system (\ref{one}) the estimates (\ref{upe}) are based on local integral
estimates of the solutions, following some ideas of \cite{BiP} for elliptic
systems with source terms. Then we use two arguments: the mean value
inequality in suitable cylinders for subsolutions of the heat equation, and an
adaptation of the bootstrap technique of \cite{BiGr}. \bigskip

\begin{notation}
For any cylinder $\tilde{Q}=\omega\times\left(  s,t\right)  \subset
\Omega\times\left(  0,T\right)  $ and any $w\in L^{1}(\tilde{Q})$ we set
\[
{\displaystyle{\int\!\!\!\!\!\!\int\!\!\!\!\!\!\!-}}_{\tilde{Q}}w=\frac
{1}{\left\vert \tilde{Q}\right\vert }\int_{s}^{t}\int_{\omega}w.
\]
For any $\rho>0$, we define the open ball $B_{\rho}=B(0,\rho)$ and the
cylinder
\[
\tilde{Q}_{\rho}=B_{\rho}\times\left[  -\rho^{2},0\right]  .
\]
We denote by $\xi_{1}$ the first eigenfunction of the Laplacian in $B_{1}$,
such that $\int_{B_{1}}\xi_{1}=1,$ with eigenvalue $\lambda_{1}$, and by $\xi$
the first eigenfunction in $B_{\rho}$ with eigenvalue $\lambda_{1,\rho
}=\lambda_{1}/\rho^{2},$ defined by
\begin{equation}
\xi(x)=\xi_{1}(\frac{x}{\rho}),\qquad\forall x\in B_{\rho}. \label{ksi}%
\end{equation}

\end{notation}

First we need a precise version of the mean value inequality.\medskip

\begin{lemma}
\label{dibe} Let $\Omega$ be any domain in $\mathbb{R}^{N}$, and let $w$ be a
subsolution of the heat equation in $\Omega\times\left(  0,T\right)  $, with
$w\in C^{2,1}(\Omega\times\left(  0,T\right)  ).$ Then for any $r>0$, there
exists a constant $C=C(N,r)$, such that for any $(x_{0},t_{0})$ and $\rho>0$
such that $(x_{0},t_{0})+\tilde{Q}_{\rho}\subset\Omega\times\left(
0,T\right)  $, and for any $\varepsilon\in\left(  0,1/2\right)  $,
\begin{equation}
\sup_{(x_{0},t_{0})+\tilde{Q}_{\rho(1-\varepsilon)}}w\leqq C\varepsilon
^{-\frac{N+2}{r^{2}}}\left(  {\displaystyle{\int\!\!\!\!\!\!\int
\!\!\!\!\!\!\!-}}_{(x_{0},t_{0})+\tilde{Q}_{\rho}}w^{r}\right)  ^{\frac{1}{r}%
}.\label{cho}%
\end{equation}

\end{lemma}

\begin{proof}
This Lemma is given in case $\varepsilon=1$ in \cite{DiB} for solutions of the
heat equation, and we adapt its proof with the parameter $\varepsilon.$ We can
assume that $(x_{0},t_{0})=0$ and $r\in\left(  0,1\right)  .$From \cite{DiB}
there exists $C_{N}=C(N)>0$ such that for any $\sigma\in\left(  0,1\right)
,$
\begin{equation}
\sup_{\tilde{Q}_{\rho\sigma}}w\leqq C_{N}(1-\sigma)^{-(N+2)}%
{\displaystyle{\int\!\!\!\!\!\!\int\!\!\!\!\!\!\!-}}_{\tilde{Q}_{\rho}%
}w.\label{num}%
\end{equation}
For any $n\in\mathbb{N}$, let $\rho_{n}=\rho(1-\varepsilon)(1+\varepsilon
/2+...+(\varepsilon/2)^{n})$, and $M_{n}=\sup_{\tilde{Q}_{\rho_{n}}}\left\vert
w\right\vert .$ From (\ref{num}) we obtain
\[
M_{n}\leqq C_{N}(1-\frac{\rho_{n}}{\rho_{n+1}})^{-(N+2)}{\displaystyle{\int
\!\!\!\!\!\!\int\!\!\!\!\!\!\!-}}_{\tilde{Q}_{\rho_{n+1}}}w;
\]
thus with a new constant $C_{N}$
\[
M_{n}\leqq C_{N}\varepsilon^{-(n+1)(N+2)}{\displaystyle{\int\!\!\!\!\!\!\int
\!\!\!\!\!\!\!-}}_{\tilde{Q}_{\rho_{n+1}}}w.
\]
From Young inequality, for any $\delta\in\left(  0,1\right)  $, we obtain
\begin{align*}
M_{n} &  \leqq C_{N}\varepsilon^{-(n+1)(N+2)}M_{n+1}^{1-r}{\displaystyle{\int
\!\!\!\!\!\!\int\!\!\!\!\!\!\!-}}_{\tilde{Q}_{\rho_{n+1}}}w^{r}\\
&  \leqq\delta M_{n+1}+r\delta^{1-1/r}(C_{N}\varepsilon^{-(n+1)(N+2)}%
)^{\frac{1}{r}}\left(  {\displaystyle{\int\!\!\!\!\!\!\int\!\!\!\!\!\!\!-}%
}_{\tilde{Q}_{\rho_{n+1}}}w^{r}\right)  ^{\frac{1}{r}}%
\end{align*}
Defining $D=$ $r\delta^{1-1/r}C_{N}{}^{\frac{1}{r}}$ and $b=\varepsilon
^{-(N+2)/r}$, we find
\[
M_{n}\leqq\delta M_{n+1}+b^{n+1}D\left(  {\displaystyle{\int\!\!\!\!\!\!\int
\!\!\!\!\!\!\!-}}_{\tilde{Q}_{\rho_{n+1}}}w^{r}\right)  ^{\frac{1}{r}}.
\]
Taking $\delta=1/2b$ and iterating, we obtain
\begin{align*}
M_{0} &  =\sup_{\tilde{Q}_{\rho(1-\varepsilon)}}\left\vert w\right\vert
\leqq\delta^{n+1}M_{n+1}+bD%
{\displaystyle\sum_{i=0}^{n}}
(\delta b)^{i}\left(  {\displaystyle{\int\!\!\!\!\!\!\int\!\!\!\!\!\!\!-}%
}_{\tilde{Q}_{\rho_{n+1}}}w^{r}\right)  ^{\frac{1}{r}}\\
&  \leqq\delta^{n+1}M_{n+1}+2bD\left(  {\displaystyle{\int\!\!\!\!\!\!\int
\!\!\!\!\!\!\!-}}_{\tilde{Q}_{\rho_{n+1}}}w^{r}\right)  ^{\frac{1}{r}}.
\end{align*}
Since $\tilde{Q}_{\rho_{n+1}}\subset\tilde{Q}_{\rho(1+\varepsilon)},$ we
deduce (\ref{cho}) by going to the limit as $n\rightarrow\infty.$ $\medskip$
\end{proof}

Next we recall a bootstrap result given from \cite[Lemma 2.2]{BiGr}:\medskip\ 

\begin{lemma}
\label{boot}Let $d,h,\ell\in\mathbb{R}$ with $d\in\left(  0,1\right)  $ and
$y,\Phi$ be two continuous positive functions on some interval $\left(
0,R\right]  .\;$Assume that there exist some $C$,$M>0$ and $\varepsilon_{0}%
\in\left(  0,1/2\right]  \;$such that, for any $\varepsilon\in\left(
0,\varepsilon_{0}\right]  $,%
\[
y(r)\leqq C\;\varepsilon^{-h}\Phi(r)\;y^{d}\left[  r(1-\varepsilon)\right]
\qquad\text{and }\max_{\tau\in\left[  r/2,r\right]  }\Phi(\tau)\leqq M\text{
}\Phi(r),
\]
or else
\[
y(r)\leqq C\;\varepsilon^{-h}\Phi(r)\;y^{d}\left[  r(1+\varepsilon)\right]
\qquad\text{and }\max_{\tau\in\left[  r,3r/2\right]  }\Phi(\tau)\leqq M\text{
}\Phi(r),
\]
for any $r\in\left(  0,R/2\right]  .\;$Then there exists another $C>0$ such
that
\[
y(r)\leqq C\;\Phi(r)^{1/(1-d)}%
\]
on $\left(  0,R/2\right]  .$\bigskip
\end{lemma}

Next we prove the estimates (\ref{upe}).\medskip\ 

\begin{proof}
[Proof of Theorem \ref{intest}]We consider any point $(x_{0},t_{0})\in
\Omega\times\left(  0,T\right)  ,$ and any $\rho>0$ such that $B(x_{0}%
,\rho)=x_{0}+B_{\rho}\subset\Omega.$ By translation we can reduce to the case
$x_{0}=0.$ For given $s\in(0,1),$ we consider a smooth function $\eta_{0}(t)$
on $\left[  -2s,0\right]  $ with values in $\left[  0,1\right]  $ such that
$\eta_{0}=1$ in $\left[  -s,0\right]  $ and $\eta_{0}(-2s)=0$ and $0\leqq
(\eta_{0})_{t}(t)\leqq Cs^{-1}$. Choosing $s$ such that $0<t_{0}-2s<t_{0},$ we
set $\eta(t)=\eta_{0}(t-t_{0}).$ We multiply the first equation in (\ref{one})
by
\[
\varphi=\xi^{\lambda}(x)\eta^{\lambda}(t),
\]
where $\xi$ is defined at (\ref{ksi}), and $\lambda>1,$ which will be chosen
large enough. We obtain
\begin{equation}
\frac{d}{dt}\left(  \int_{B_{\rho}}u\xi^{\lambda}\eta^{\lambda}(t)\right)
+\int_{B_{\rho}}v^{p}\xi^{\lambda}\eta^{\lambda}=\lambda\int_{B_{\rho}}%
u\xi^{\lambda}\eta^{\lambda-1}\eta_{t}(t)+\int_{B_{\rho}}u(\Delta\xi^{\lambda
})\eta^{\lambda}. \label{truc}%
\end{equation}
By computation, we find
\[
\rho^{2}\Delta\xi^{\lambda}(x)=\Delta\xi_{1}^{\lambda}(\frac{x}{\rho
})=-\lambda\lambda_{1}\xi_{1}^{\lambda}(\frac{x}{\rho})+\lambda(\lambda
-1)\xi_{1}^{\lambda-2}\left\vert \nabla\xi_{1}\right\vert ^{2}(\frac{x}{\rho
}).
\]
For given $\ell>1$, if $\lambda>2\ell^{\prime}$, the function $g_{\ell}%
(y)=\xi_{1}^{\lambda/\ell^{\prime}-2}\left\vert \nabla\xi_{1}\right\vert ^{2}$
is bounded, thus%

\begin{align*}
\int_{B_{\rho}}u(.,t)(\Delta\xi^{\lambda})\eta^{\lambda}(t)  &  \leqq
\frac{\lambda(\lambda-1)}{\rho^{2}}\int_{B_{\rho}}u(x,t)(\xi_{1}^{\lambda
-2}\left\vert \nabla\xi_{1}\right\vert ^{2})(\frac{x}{\rho})\eta^{\lambda
}(t)dx\\
&  =\frac{\lambda(\lambda-1)}{\rho^{2}}\int_{B_{\rho}}u(x,t)\xi^{\lambda/\ell
}g_{\ell}(\frac{x}{\rho})\eta^{\lambda}(t)dx\\
&  \leqq\frac{\lambda(\lambda-1)}{\rho^{2}}\left(  \int_{B_{\rho}}%
u(.,t)^{\ell}\xi^{\lambda}\eta^{\lambda}(t)\right)  ^{1/\ell}\left(
\int_{B_{\rho}}g_{\ell}^{\ell^{\prime}}(\frac{x}{\rho})\eta^{\lambda
}(t)dx\right)  ^{1/\ell^{\prime}}\\
&  \leqq C\rho^{N/\ell^{\prime}-2}\left(  \int_{B_{\rho}}u(.,t)^{\ell}%
\xi^{\lambda}\eta^{\lambda}(t)\right)  ^{1/\ell}%
\end{align*}
and even with different constants $C=C(N,\ell)$%
\begin{align}
\int_{B_{\rho}}u(.,t)\left\vert \Delta\xi^{\lambda}\right\vert \eta^{\lambda
}(t)  &  \leqq\lambda\lambda_{1}\rho^{-2}\int_{B_{\rho}}u(.,t)\xi^{\lambda
}\eta^{\lambda}(t)+C\rho^{N/\ell^{\prime}-2}\left(  \int_{B_{\rho}%
}u(.,t)^{\ell}\xi^{\lambda}\eta^{\lambda}(t)\right)  ^{1/\ell}\nonumber\\
&  \leqq C\rho^{N/\ell^{\prime}-2}\left(  \int_{B_{\rho}}u(.,t)^{\ell}%
\xi^{\lambda}\eta^{\lambda}(t)\right)  ^{1/\ell}. \label{veri}%
\end{align}
Moreover
\begin{align*}
&  \int_{B_{\rho}}u(.,t)\xi^{\lambda}\eta^{\lambda-1}\eta_{t}(t)\leqq
Cs^{-1}\left(  \int_{B_{\rho}}u(.,t)^{\ell}\xi^{\lambda}\eta^{\lambda
}(t)\right)  ^{1/\ell}\left(  \int_{B_{\rho}}\xi^{\lambda}\eta^{\lambda
-\ell^{\prime}}(t)\right)  ^{1/\ell^{\prime}}\\
&  \leqq C\rho^{N/\ell^{\prime}}s^{-1}\left(  \int_{B_{\rho}}u(.,t)^{\ell}%
\xi^{\lambda}\eta^{\lambda}(t)\right)  ^{1/\ell}.
\end{align*}
Integrating (\ref{truc}) on $\left(  t_{0}-2s,t_{0}\right)  ,$ and using
H\"{o}lder inequality,
\begin{align}
\int_{B_{\rho}}u(.,t_{0})\xi^{\lambda}+\int_{t_{0}-2s}^{t_{0}}\int_{B_{\rho}%
}v^{p}\xi^{\lambda}\eta^{\lambda}  &  \leqq C\rho^{N/\ell^{\prime}}(\rho
^{-2}+s^{-1})\int_{t_{0}-2s}^{t_{0}}\left(  \int_{B_{\rho}}u^{\ell}%
\xi^{\lambda}\eta^{\lambda}\right)  ^{1/\ell}\nonumber\\
&  \leqq C\rho^{N/\ell^{\prime}}(\rho^{-2}+s^{-1})s^{1/\ell^{\prime}}\left(
\int_{t_{0}-2s}^{t_{0}}\int_{B_{\rho}}u^{\ell}\xi^{\lambda}\eta^{\lambda
}\right)  ^{1/\ell}. \label{onu}%
\end{align}
In the same way, for any $\kappa>1,$ if $\lambda>2k^{\prime},$
\begin{equation}
\int_{B_{\rho}}v(.,t_{0})\xi^{\lambda}+\int_{t_{0}-2s}^{t_{0}}\int_{B_{\rho}%
}u^{q}\xi^{\lambda}\eta^{\lambda}\leqq C\rho^{N/\kappa^{\prime}}(\rho
^{-2}+s^{-1})s^{1/\kappa^{\prime}}(\int_{t_{0}-2s}^{t_{0}}\int_{B_{\rho}%
}v^{\kappa}\xi^{\lambda}\eta^{\lambda})^{1/\kappa}. \label{onv}%
\end{equation}
Next we discuss according to the values of $p$ and $q.\bigskip$

\noindent\textbf{First case: }$p,q>1.$ We take $\ell=q,\kappa=p,$ and
$2s=\rho^{2}$ and consider any $t_{0}$ such that $0<t_{0}-\rho^{2}<t_{0}<T.$
Let us denote $Q_{\rho}=(0,t_{0})+\tilde{Q}_{\rho}.$ Then
\begin{align*}%
{\displaystyle\iint_{Q_{\rho}}}
v^{p}\xi^{\lambda}\eta^{\lambda}  &  \leqq C\rho^{(N+2)/q^{\prime}-2}\left(
{\displaystyle\iint_{Q_{\rho}}}
u^{q}\xi^{\lambda}\eta^{\lambda}\right)  ^{1/q},\\
{\displaystyle\iint_{Q_{\rho}}} u^{q}\xi^{\lambda}\eta^{\lambda}  &  \leqq
C\rho^{(N+2)/p^{\prime}-2}\left(  {\displaystyle\iint_{Q_{\rho}}} v^{p}%
\xi^{\lambda}\eta^{\lambda}\right)  ^{1/p},
\end{align*}
that means
\begin{equation}
{\displaystyle{\int\!\!\!\!\!\!\int\!\!\!\!\!\!\!-}} _{Q_{\rho}}v^{p}%
\xi^{\lambda}\eta^{\lambda}\leqq C\rho^{-2}\left(  {\displaystyle{\int
\!\!\!\!\!\!\int\!\!\!\!\!\!\!-}} _{Q_{\rho}}u^{q}\xi^{\lambda}\eta^{\lambda
}\right)  ^{1/q,} \label{jjj}%
\end{equation}%
\[
{\displaystyle{\int\!\!\!\!\!\!\int\!\!\!\!\!\!\!-}} _{Q_{\rho}}u^{q}%
\xi^{\lambda}\eta^{\lambda}\leqq C\rho^{-2}\left(  {\displaystyle{\int
\!\!\!\!\!\!\int\!\!\!\!\!\!\!-}} _{Q_{\rho}}v^{p}\xi^{\lambda}\eta^{\lambda
}\right)  ^{1/p}.
\]
Hence%
\[
{\displaystyle{\int\!\!\!\!\!\!\int\!\!\!\!\!\!\!-}} _{Q_{\rho}}u^{q}%
\xi^{\lambda}\eta^{\lambda}\leqq C\rho^{-2(p+1)/p}\left(  {\displaystyle{\int
\!\!\!\!\!\!\int\!\!\!\!\!\!\!-}} _{Q_{\rho}}u^{q}\xi^{\lambda}\eta^{\lambda
}\right)  ^{1/pq}.
\]
Then we get an estimate of the form
\begin{equation}
({\displaystyle{\int\!\!\!\!\!\!\int\!\!\!\!\!\!\!-}} _{Q_{\rho/2}}%
u^{q})^{1/q}\leqq\frac{C}{\rho^{2(p+1)/(pq-1)}} \label{shi}%
\end{equation}
and similarly
\begin{equation}
({\displaystyle{\int\!\!\!\!\!\!\int\!\!\!\!\!\!\!-}} _{Q_{\rho/2}}%
v^{p})^{1/p}\leqq\frac{C}{\rho^{2(q+1)/(pq-1)}} \label{shy}%
\end{equation}
But $u$ is a subsolution of the heat equation, hence there exists a $C=C(N,q)
$ such that
\[
u(x,t)\leqq C({\displaystyle{\int\!\!\!\!\!\!\int\!\!\!\!\!\!\!-}}
_{Q_{\rho/2}}u^{q})^{1/q},
\]
from Lemma \ref{dibe} with $r=q$ and $\varepsilon=1.$ Taking $\rho^{2}%
=t_{0}/2M$, with $M>1,$ we deduce the estimates
\[
u(x,t)\leqq\frac{C}{t^{(p+1)/(pq-1)}},\qquad v(x,t)\leqq\frac{C}%
{t^{(q+1)/(pq-1)}},
\]
for any $t\in\left(  0,T\right)  $ and any $x\in\Omega$ such that
$B(x,\sqrt{t/2M})\subset\Omega,$ with $C=C(N,p,q,M).$Then (\ref{upe})
follows.\medskip

\noindent\textbf{General case: }$pq>1.$ We can assume $p\leqq1<q.$ Taking
again $0<t_{0}-\rho^{2}<t_{0}<T$ and $2s=\rho^{2},$ and using (\ref{onu}) with
$\ell=q>1,$ we find again (\ref{jjj}). Using (\ref{onv}), we find for any
$\kappa>1$,
\begin{equation}%
{\displaystyle\iint_{Q_{\rho}}}
u^{q}\xi^{\lambda}\eta^{\lambda}\leqq C\rho^{(N+2)/\kappa^{\prime}-2}\left(
{\displaystyle\iint_{Q_{\rho}}}
v^{\kappa}\xi^{\lambda}\eta^{\lambda}\right)  ^{1/\kappa}\leqq C\rho
^{(N+2)/\kappa^{\prime}-2}\sup_{Q_{\rho}}v^{1-p/\kappa}\left(
{\displaystyle\iint_{Q_{\rho}}}
v^{p}\right)  ^{1/\kappa}. \label{fob}%
\end{equation}
More precisely, for any $\varepsilon\in\left(  0,1/2\right)  $, from Lemma
\ref{dibe}, we find taking $r=p$ and $\kappa=q,$
\[
\sup_{Q_{\rho}}v\leqq C\varepsilon^{-(N+2)/p^{2}}\rho^{-(N+2)/p}\left(
{\displaystyle\iint_{Q_{\rho(1+\varepsilon)}}}
v^{p}\right)  ^{1/p},
\]
then
\begin{align*}
\sup_{Q_{\rho}}v^{1-p/q}\left(
{\displaystyle\iint_{Q_{\rho}}}
v^{p}\right)  ^{1/q}  &  \leqq C\varepsilon^{-(N+2)\frac{(q-p)}{p^{2}q}}%
\rho^{-(N+2)\frac{(q-p)}{pq}}\left(
{\displaystyle\iint_{Q_{\rho(1+\varepsilon)}}}
v^{p}\right)  ^{\frac{(q-p)}{pq}+\frac{1}{q}}\\
&  =C\varepsilon^{-(N+2)\frac{(q-p)}{p^{2}q}}\rho^{-(N+2)\frac{(q-p)}{pq}%
}\left(
{\displaystyle\iint_{Q_{\rho(1+\varepsilon)}}}
v^{p}\right)  ^{1/p}.
\end{align*}
Using (\ref{fob}) we deduce%

\begin{align*}%
{\displaystyle\iint_{Q_{\rho(1-\varepsilon)}}}
u^{q}  &  \leqq C\varepsilon^{-(2\lambda+(N+2)\frac{(q-p)}{p^{2}q})}%
\rho^{(N+2)/q^{\prime}-2-(N+2)\frac{(q-p)}{pq}}\left(
{\displaystyle\iint_{Q_{\rho(1+\varepsilon)}}}
v^{p}\right)  ^{1/p}\\
&  =C\varepsilon^{-(2\lambda+(N+2)\frac{(q-p)}{p^{2}q})}\rho^{(N+2)/(1-1/p)-2}%
\left(
{\displaystyle\iint_{Q_{\rho(1+\varepsilon)}}}
v^{p}\right)  ^{1/p};
\end{align*}
setting $h=2\lambda+(N+2)(q-p)/p^{2}q,$ that means
\begin{equation}
{\displaystyle{\int\!\!\!\!\!\!\int\!\!\!\!\!\!\!-}}_{Q_{\rho(1-\varepsilon)}%
}u^{q}\leqq C\varepsilon^{-h}\rho^{-2}\left(  {\displaystyle{\int
\!\!\!\!\!\!\int\!\!\!\!\!\!\!-}}_{Q_{\rho(1+\varepsilon)}}v^{p}\right)
^{1/p}. \label{bla}%
\end{equation}
\noindent Next from (\ref{jjj}) we have
\begin{equation}
{\displaystyle{\int\!\!\!\!\!\!\int\!\!\!\!\!\!\!-}}_{Q_{\rho(1-\varepsilon)}%
}v^{p}\leqq C\rho^{-2}\left(  {\displaystyle{\int\!\!\!\!\!\!\int
\!\!\!\!\!\!\!-}}_{Q_{\rho(1+\varepsilon)}}u^{q}\right)  ^{1/q}, \label{ggg}%
\end{equation}
thus changing $\rho(1-\varepsilon)$ into $\rho(1+\varepsilon),$
\[
{\displaystyle{\int\!\!\!\!\!\!\int\!\!\!\!\!\!\!-}}_{Q_{\rho(1+\varepsilon)}%
}v^{p}\leqq C\rho^{-2}\left(  {\displaystyle{\int\!\!\!\!\!\!\int
\!\!\!\!\!\!\!-}}_{Q_{\rho(1+4\varepsilon)}}u^{q}\right)  ^{1/q}.
\]
Hence from (\ref{bla}), we deduce
\[
{\displaystyle{\int\!\!\!\!\!\!\int\!\!\!\!\!\!\!-}}_{Q_{\rho(1-\varepsilon)}%
}u^{q}\leqq C\varepsilon^{-h}\rho^{-2(p+1)/p}\left(  {\displaystyle{\int
\!\!\!\!\!\!\int\!\!\!\!\!\!\!-}}_{Q_{\rho(1+4\varepsilon)}}u^{q}\right)
^{1/pq}.
\]
From Lemma \ref{boot}, we conclude that
\[
\left(  {\displaystyle{\int\!\!\!\!\!\!\int\!\!\!\!\!\!\!-}}_{Q_{\rho}}%
u^{q}\right)  ^{(pq-1)/q}\leqq C\rho^{-2(p+1)}.
\]
Hence (\ref{shi}) follows as above, and then (\ref{shy}) from (\ref{ggg}), and
the conclusion follows again.\medskip
\end{proof}

Next we give a first extension of the scalar estimate (\ref{mai}) to system
(\ref{one}), using some ideas of \cite[p. 243]{BiGr}.\medskip

\begin{proposition}
\label{set}Let $q\geqq p>1.$ Let $(u,v)$ be any positive solution of system
(\ref{one}) in $\Omega\times\left(  0,T\right)  ,$ where $\Omega$ is a bounded
$C^{2}$ domain Then there exists a constant $C=C(N,p,q)$ such that
\begin{equation}
u^{(q+1)/(p+1)}(x,t)+v(x,t)\leqq C(t+d^{2}(x,\partial\Omega))^{-1/(p-1)}%
,\qquad\forall(x,t)\in\Omega\times\left(  0,T\right)  \label{sim}%
\end{equation}
\medskip
\end{proposition}

\begin{proof}
Let $F=(k+u)^{d}+v$, with $d=(q+1)/(p+1)>1$ and $k>0.$ Then
\begin{align*}
F_{t}-\Delta F  &  =d(k+u)^{d-1}(u_{t}-\Delta u)-d(d-1)(k+u)^{d-2}\left\vert
\nabla u\right\vert ^{2}+v_{t}-\Delta v\\
&  \leqq-d(k+u)^{d-1}v^{p}-u^{q}.
\end{align*}
But $(k+u)^{q}\leqq2^{q-1}(k^{q}+u^{q})$, thus
\[
F_{t}-\Delta F+d(k+u)^{d-1}v^{p}+2^{1-q}(k+u)^{q}\leqq k^{q}.
\]
Observe that $(k+u)^{q}=(k+u)^{d-1}(k+u)^{dp}$, and $F^{p}\leqq2^{p-1}%
((k+u)^{dp}+v^{p})$. Then
\[
F_{t}-\Delta F+c(k+u)^{d-1}F^{p}\leqq k^{q},
\]
with $c=2^{1-p}\min(d,2^{1-q}).$ In particular, taking $k=c^{-1/(d-1)},$ $F$
is a subsolution of equation
\begin{equation}
U_{t}-\Delta U+U^{p}=K \label{uk}%
\end{equation}
in $\Omega\times\left(  0,T\right)  ,$ where $K=k^{q}=K(p,q).$ Let
$f(t)=\left(  (p-1)t)\right)  ^{-1/(Q-1)}$ and let $g$ be the maximal solution
of the stationary problem $-\Delta U+U^{p}=0$ in $\Omega$ such that $g=\infty$
on $\partial\Omega.$ Then for any $\varepsilon>0,$ the function $(x,t)\mapsto
G_{\varepsilon}(x,t)=K^{1/p}+f(t-\varepsilon)+g(x)$ is a supersolution of
equation (\ref{uk}) in $\Omega\times\left(  \varepsilon,T\right)  .$ Going to
the limit as $\varepsilon\longrightarrow0,$ it follows that
\[
F(x,t)\leqq K^{1/p}+f(t)+g(x)
\]
in $\Omega\times\left(  0,T\right)  ;$ then there exists a constants
$C^{\prime}=C^{\prime}(N,p)$ such that
\[
F(x,t)\leqq K^{1/p}+f(t)+C^{\prime}d(x,\partial\Omega)^{-2/(p-1)}%
,\qquad\forall(x,t)\in\Omega\times\left(  0,T\right)  ,
\]
and the conclusion follows.\medskip
\end{proof}

\textbf{Open problem: } The estimate (\ref{sim}) does not appear to be
optimal, except in the case $p=q$ where $u=v$ is a solution of the scalar
equation (\ref{sca}). Can we obtain for $p,q>1$, and even for $pq>1$, an upper
estimate in $\Omega\times\left(  0,T\right)  $ of the form
\[
u(x,t)\leqq C(t+d^{2}(x,\partial\Omega))^{-a},\qquad v(x,t)\leqq
C(t+d^{2}(x,\partial\Omega))^{-b},
\]
with $C=C(N,p,q)?$

\section{Initial trace}

First we show some properties available for any $p,q>0.$\medskip

\begin{lemma}
\label{chou} Assume $p,q>0.$ Let $(u,v)$ be any positive solution of system
(\ref{one}), and let $B(x_{0},\rho)\subset\Omega.$ If $\int_{0}^{T}%
\int_{B(x_{0},\rho)}v^{p}<\infty$, then $\int_{B(x_{0},\bar{\rho})}u(.,t)$ is
bounded as $t\rightarrow0$ for any $\bar{\rho}<\rho,$ and there exists a Radon
measure $m_{1,\rho}$ on $B(x_{0},\rho)$ such that for any $\psi\in
C_{c}^{\infty}(B(x_{0},\rho)),$
\[
\lim_{t\rightarrow0}\int_{B(x_{0},\rho)}u(.,t)\psi=m_{1,\rho}(\psi).
\]

\end{lemma}

\begin{proof}
\noindent We reduce to the case $x_{0}=0.$ We set
\begin{equation}
X(t)=\int_{B_{\rho}}u(.,t)\xi^{\lambda},\quad Y(t)=\int_{B_{\rho}}%
v(.,t)\xi^{\lambda},\quad Z(t)=\int_{B_{\rho}}u^{q}(.,t)\xi^{\lambda},\quad
W(t)=\int_{B_{\rho}}v^{p}(.,t)\xi^{\lambda}. \label{xw}%
\end{equation}
where $\xi$ is defined at (\ref{ksi}) and $\lambda\geqq2.$ We obtain
\begin{align*}
X_{t}+W  &  =\frac{d}{dt}\left(  \int_{B_{\rho}}u\xi^{\lambda}\right)
+\int_{B_{\rho}}v^{p}\xi^{\lambda}=\int_{B_{\rho}}u(\Delta\xi^{\lambda})\\
&  =-\lambda\lambda_{1,\rho}\int_{B_{\rho}}u\xi^{\lambda}+\lambda
(\lambda-1)\int_{B_{\rho}}u\xi^{\lambda-2}\left\vert \nabla\xi\right\vert
^{2}\\
&  \geqq-\lambda\lambda_{1,\rho}\int_{B_{\rho}}u\xi^{\lambda}=-\lambda
\lambda_{1,\rho}X,
\end{align*}
hence
\[
\frac{d}{dt}(e^{\lambda\lambda_{1,\rho}t}X(t))+e^{\lambda\lambda_{1,\rho}%
t}W(t)\geqq0.
\]
By integration we obtain for any $t<\theta$
\[
e^{\lambda\lambda_{1,\rho}\theta}X(\theta)-e^{\lambda\lambda_{1,\rho}%
t}X(t)+\int_{t}^{\theta}e^{\lambda\lambda_{1,\rho}s}W(s)ds\geqq0;
\]
and from our assumption $W\in L^{1}(\left(  0,T\right)  ).$ Then
$e^{\lambda\lambda_{1,\rho}t}X(t)$ is bounded, and in turn $X(t)$ is bounded.
Then $\int_{B_{\rho}}u(.,t)\xi^{\lambda}$ is bounded, hence $\int
_{B(x_{0},\bar{\rho})}u(.,t)$ is bounded. Let $\psi\in C_{c}^{\infty}%
(B(x_{0},\rho)).$ Then
\[
\frac{d}{dt}\int_{B_{\rho}}u(.,t)\psi+\int_{B_{\rho}}v^{p}\psi=\int_{B_{\rho}%
}u(\Delta\psi).
\]
Since $\Delta\psi$ is bounded with compact support, we have $\left\vert
\psi\right\vert +\left\vert \Delta\psi\right\vert \leqq C\xi^{\lambda}$ for
some positive constant $C$, and thus $\int_{B_{\rho}}u(\Delta\psi)$ is
bounded, implying $\int_{B_{\rho}}u(.,t)\psi$ has a limit $m_{1,\rho}(\psi),$
which defines a Radon measure $m_{1,\rho}$ on $B_{\rho}$.\medskip
\end{proof}

\begin{lemma}
\label{doux} Assume $p,q>0.$ Let $(u,v)$ be any positive solution of system
(\ref{one}), and let $B(x_{0},\rho_{0})\subset\Omega.$ If $\int_{B(x_{0}%
,\rho_{0})}u(.,t)$ is bounded as $t\rightarrow0$, then

(i) for any $\rho<\rho_{0},$ $\int_{t}^{\theta}\int_{B_{\rho}}v^{p}$ is bounded;

(ii) for any $\bar{\rho}<\rho_{0},$ any $1\leqq\sigma<1+2/N,$ and any
$0<t<\theta<T$
\begin{equation}
\int_{t}^{\theta}\int_{B_{\bar{\rho}}}u^{\sigma}dx\leqq C, \label{pol}%
\end{equation}
where $C=C(N,p,q,\bar{\rho},\rho_{0},\sigma).$\bigskip
\end{lemma}

\begin{proof}
We still reduce to the case $x_{0}=0.$

(i) Let $0<t<\theta<T$ with fixed $\theta,$ and $C=\sup_{\left(
0,\theta\right]  }$ $\int_{B_{\rho_{0}}}u(.,t).$ Let $\psi\in C_{c}^{\infty
}(B_{\rho_{0}})$ with values in $\left[  0,1\right]  $ such that $\psi=1$ on
$B_{\rho}$. Taking $\psi$ as a test function in the equation in $u$ and
integrating between $t$ and $\theta,$ we find
\[
\frac{d}{dt}\left(  \int_{B_{\rho_{0}}}u\psi\right)  +\int_{B_{\rho_{0}}}%
v^{p}\psi=\int_{B_{\rho_{0}}}u(\Delta\psi)\leqq C\left\Vert \Delta
\psi\right\Vert _{L^{\infty}(\Omega)},
\]
thus
\[
\int_{B_{\rho_{0}}}u(.,\theta)\psi+\int_{t}^{\theta}\int_{B_{\rho_{0}}}%
v^{p}\psi\leqq C(\left\Vert \Delta\psi\right\Vert _{L^{\infty}(\Omega)}+1),
\]
hence $\int_{t}^{\theta}\int_{B_{\rho}}v^{p}$ is bounded. $\medskip$

(ii) Here we use the ideas of \cite[Propositions 2.1,2.2.]{BiCVe} relative to
quasilinear equations in order to estimate the gradient. Since $\sigma<1+2/N$,
we can fix $\alpha=\alpha(\sigma)$ such that
\begin{equation}
-1<\alpha<0\text{ and }\sigma\leqq\alpha+1+2/N. \label{alp}%
\end{equation}
Let $\rho$ be fixed such that $\bar{\rho}<\rho<\rho_{0}.$ We multiply the
equation in $u$ by $(1+u)^{\alpha}\xi^{\lambda}$, where $\xi$ is defined at
(\ref{ksi}), with $\lambda\geqq2/\left\vert \alpha\right\vert .$ Then we find
for fixed $\theta<T,$ and any $0<t\leqq\theta$
\begin{align*}
&  \frac{1}{\alpha+1}\int_{B_{\rho}}(1+u(.,t))^{\alpha+1}\xi^{\lambda
}+\left\vert \alpha\right\vert
{\displaystyle\int_{t}^{\theta}}
\int_{B_{\rho}}(1+u)^{\alpha-1}\left\vert \nabla u\right\vert ^{2}\xi
^{\lambda}\\
&  =\frac{1}{\alpha+1}\int_{B_{\rho}}(1+u(.,\theta))^{\alpha+1}\xi^{\lambda}+%
{\displaystyle\int_{t}^{\theta}}
\int_{B_{\rho}}v^{p}(1+u)^{\alpha}\xi^{\lambda}+\lambda%
{\displaystyle\int_{t}^{\theta}}
\int_{B_{\rho}}(1+u)^{\alpha}\xi^{\lambda-1}\nabla u.\nabla\xi.
\end{align*}
Applying twice the H\"{o}lder inequality, we find
\begin{align}
&  \frac{1}{\alpha+1}\int_{B_{\rho}}(1+u(.,t))^{\alpha+1}\xi^{\lambda}%
+\frac{1}{2}\left\vert \alpha\right\vert
{\displaystyle\int_{t}^{\theta}}
\int_{B_{\rho}}(1+u)^{\alpha-1}\left\vert \nabla u\right\vert ^{2}\xi
^{\lambda}\nonumber\\
&  \leqq C+%
{\displaystyle\int_{t}^{\theta}}
\int_{B_{\rho}}v^{p}(1+u)^{\alpha}\xi^{\lambda}+C%
{\displaystyle\int_{t}^{\theta}}
\int_{B_{\rho}}(1+u)^{\alpha+1}\xi^{\lambda-2}\left\vert \nabla\xi\right\vert
^{2}\nonumber\\
&  \leqq C+%
{\displaystyle\int_{t}^{\theta}}
\int_{B_{\rho}}v^{p}+C\left(
{\displaystyle\int_{t}^{\theta}}
\int_{B_{\rho}}(1+u)\xi^{\lambda}\right)  ^{1+\alpha}\left(
{\displaystyle\int_{t}^{\theta}}
\int_{B_{\rho}}\xi^{\lambda-2/\left\vert \alpha\right\vert }\left\vert
\nabla\xi\right\vert ^{2/\left\vert \alpha\right\vert }\right)  ^{\left\vert
\alpha\right\vert }, \label{tsf}%
\end{align}
where $C$ depends on $\theta$ and $\sigma.$ Since $\int_{B_{\rho}}%
u(.,t)\xi^{\lambda}$ is bounded, and $\int_{t}^{\theta}\int_{B_{\rho}}v^{p}$
is bounded, we obtain an estimate of the gradient:
\[%
{\displaystyle\int_{t}^{\theta}}
\int_{B_{\rho}}(1+u)^{\alpha-1}\left\vert \nabla u\right\vert ^{2}\xi
^{\lambda}\leqq C.
\]
Next recall the Gagliardo-Nirenberg estimate: let $\,m\geq1,\gamma\in\left[
1,+\infty\right)  $ and $\nu\in\left[  0,1\right]  $ such that
\begin{equation}
\frac{1}{\gamma}=\nu(\frac{1}{2}-\frac{1}{N})+\frac{1-\nu}{m}; \label{re}%
\end{equation}
then there exists $C=C(N,m,\nu,\rho)>0$ such that for any $w$ $\in
W^{1,2}(B_{\bar{\rho}})\cap L^{m}(B_{\bar{\rho}})$,
\begin{equation}
\left\Vert w-\overline{w}\right\Vert _{L^{\gamma}(B_{\bar{\rho}})}\leq
C\left\Vert \left\vert \nabla w\right\vert \right\Vert _{L^{2}(B_{\bar{\rho}%
})}^{\nu}\left\Vert w-\overline{w}\right\Vert _{L^{m}(U)}^{1-\nu}. \label{gn}%
\end{equation}
We apply it to $w(x,t)=(1+u(x,t))^{\beta}$, and
\begin{equation}
\beta=\frac{1+\alpha}{2},\qquad\gamma=2+\frac{2}{N\beta},\qquad\nu=\frac
{2}{\gamma},\qquad m=\frac{1}{\beta}, \label{gate}%
\end{equation}
which satisfy (\ref{re}). Therefore, for any $t\in\left(  0,\theta\right)  ,$
\begin{align*}
\int_{B_{\bar{\rho}}}\left\vert (1+u(.,t))^{\beta}-\overline{w}(t)\right\vert
^{\gamma}  &  \leq C\left(  \int_{B_{\bar{\rho}}}(1+u(.,t))^{\alpha
-1}\left\vert \nabla u(.,t)\right\vert ^{2}\right)  \times\\
&  \left(  \int_{B_{\bar{\rho}}}\left\vert (1+u(.,t))^{\beta}-\overline
{w}(t)\right\vert ^{1/\beta}\right)  ^{(1-\nu)\gamma\beta}.
\end{align*}
Now $\left\Vert \overline{w}(.)\right\Vert _{L^{\infty}(\left(  0,\theta
\right)  )}\leq C$ because $\beta\in(0,1)$ and $\int_{B_{\rho}}u(.,t)$ is
bounded; in turn we get
\[
\int_{B_{\bar{\rho}}}\left\vert (1+u(x,t))^{\beta}-\overline{w}(t)\right\vert
^{1/\beta}dx\leq C,
\]
Therefore,
\[
\int_{B_{\bar{\rho}}}(1+u(x,t))^{\beta\gamma}\;dx\leq C\int_{B_{\bar{\rho}}%
}(1+u(.,t))^{\alpha-1}\left\vert \nabla u(.,t)\right\vert ^{2}dx+C.
\]
Integrating on $\left(  0,\theta\right)  $ we obtain
\[
\int_{0}^{\theta}\int_{B_{\bar{\rho}}}(1+u(t))^{\beta\gamma}\;dx<C.
\]
Observing that $\beta\gamma=\alpha+1+2/N$, and $\alpha$ is defined by
(\ref{alp}) we conclude to (\ref{pol}).\medskip
\end{proof}

In order of proving Theorem \ref{trace} we show the following dichotomy
property:\medskip

\begin{proposition}
\label{dic} Assume $p,q>1.$ Let $(u,v)$ be a positive solution of the system
in $\Omega\times\left(  0,T\right)  .$ Let $x_{0}\in\Omega.$ Then the
following alternative holds:

(i) Either there exists a ball $B(x_{0},\rho)\subset\Omega$ such that
$\int_{0}^{T}\int_{B(x_{0},\rho)}(u^{q}+v^{p})<\infty$ and two Radon measures
$m_{1,\rho}$ and $m_{2,\rho}$ on $B(x_{0}$,$\rho)$, such that for any $\psi\in
C_{c}^{0}(B(x_{0},\rho)\mathcal{)}$,%
\begin{equation}
\lim_{t\rightarrow0}\int_{B(x_{0},\rho)}u(.,t)\psi=\int_{B(x_{0},\rho)}\psi
dm_{1,\rho},\qquad\lim_{t\rightarrow0}\int_{B(x_{0},\rho)}v(.,t)\psi
=\int_{B(x_{0},\rho)}\psi dm_{2,\rho}, \label{tru}%
\end{equation}

(ii) Or for any ball $B(x_{0},\rho)\subset\Omega$ there holds $\int_{0}%
^{T}\int_{B(x_{0},\rho)}(u^{q}+v^{p})=\infty$ and then
\begin{equation}
\lim_{t\rightarrow0}\int_{B(x_{0},\rho)}(u(.,t)+v(.,t))=\infty. \label{sum}%
\end{equation}
\bigskip
\end{proposition}

\begin{proof}
(i) Assume that there exists a ball $B(x_{0},\rho)\subset\Omega$ such that
$\int_{0}^{T}\int_{B(x_{0},\rho)}(u^{q}+v^{p})<\infty.$ Then (\ref{tru})
follows from Lemma \ref{chou}.

(ii) Suppose that for any ball $\int_{0}^{T}\int_{B(x_{0},\rho)}(u^{q}%
+v^{p})=\infty$. Consider a fixed $\rho>0$ such that $B(x_{0},\rho).$ We can
assume $x_{0}=0.\ $We choose the test function $\xi^{\lambda}$, where $\xi$ is
defined at (\ref{ksi}) and $\lambda>2\max(p^{\prime},q^{\prime})$. Then
\[
\frac{d}{dt}\left(  \int_{B_{\rho}}u\xi^{\lambda}\right)  +\int_{B_{\rho}%
}v^{p}\xi^{\lambda}=\int_{B_{\rho}}u(\Delta\xi^{\lambda}).
\]
As above from (\ref{veri}), since $\lambda$ is large enough,
\[
\int_{B_{\rho}}u\left\vert \Delta\xi^{\lambda}\right\vert \leqq C(\int
_{B_{\rho}}u^{q}\xi^{\lambda})^{1/q},
\]
where $C$ depends on $\rho.$ Let $0<t<\theta<T.$ Consider $X,Y,Z,W$ defined by
(\ref{xw}). Then we find with new constants $C>0\quad$
\begin{align*}
X_{t}(t)+W(t)  &  \leqq CZ^{1/q}(t)\leqq\frac{Z(t)}{2}+C,\\
Y_{t}(t)+Z(t)  &  \leqq CW^{1/p}(t)\leqq\frac{W(t)}{2}+C.
\end{align*}
By addition
\[
(X+Y)_{t}(t)+\frac{Z+W}{2}(t)\leqq C
\]
By hypothesis $Z+W\not \in L^{1}(\left(  0,T)\right)  $, then
\[
\lim_{t\rightarrow0}(X(t)+Y(t))=\lim_{t\rightarrow0}\int_{B_{\rho}%
}(u(.,t)+v(.,t))\xi^{\lambda}=\infty,
\]
thus
\begin{equation}
\lim_{t\rightarrow0}\int_{B_{\bar{\rho}}}(u(.,t)+v(.,t))=\infty\label{film}%
\end{equation}
for any $\bar{\rho}<\rho,$ and the conclusion follows, since $\rho$ is arbitrary.
\end{proof}

As a direct consequence we deduce Theorem \ref{trace}.\medskip

\begin{proof}
[Proof of Theorem \ref{trace}]Let
\[
\mathcal{R=}\left\{  x_{0}\in\Omega:\exists\rho>0,\quad B(x_{0},\rho
)\subset\Omega,\quad\lim\sup\int_{B(x_{0},\rho)}(u(.,t)+v(.,t)<\infty\right\}
,
\]
and $\mathcal{S}=\Omega\backslash\mathcal{R}$. Then $\mathcal{R}$ is open, and
from proposition \ref{dic}, there exists unique Radon measures $\mu_{1}%
,\mu_{2}$ on $\mathcal{R}$ such that (\ref{ccc}) holds, and (\ref{sum})
implies (\ref{ddd}) on any open set $\mathcal{U}$ such that $\mathcal{U}%
\cap\mathcal{S\neq\emptyset}.$\medskip
\end{proof}

Next we give more information when $p,q$ are subcritical.\medskip

\begin{proposition}
\label{comp}Assume $0<p,q<1+2/N.$ Let $(u,v)$ be a positive solution of the
system in $\Omega\times\left(  0,T\right)  .$ Let $x_{0}\in\Omega.$ Then then
the eventuality (ii) of Theorem \ref{dic} is equivalent to:\bigskip

(iii) for any ball $B(x_{0}$,$\rho)\subset\Omega$ there holds
\begin{equation}
\int_{0}^{T}\int_{B(x_{0},\rho)}u^{q}=\infty\text{ and }\int_{0}^{T}%
\int_{B(x_{0},\rho)}v^{p}=\infty. \label{fru}%
\end{equation}

\end{proposition}

\begin{proof}
It is clear that (iii) implies (ii). Suppose that (iii) does not hold, and
reduce to $x_{0}=0$. Then there exists a ball $B_{\rho}$ such that for example%
\[
\int_{0}^{T}\int_{B_{\rho}}v^{p}<\infty.
\]
Then for any $\bar{\rho}<\rho$ $\int_{B_{\bar{\rho}}}u(.,t)$ is bounded as
$t\rightarrow0,$ from Lemma (\ref{chou}). Since $q<1+2/N,$ we obtain
\[
\int_{t}^{\theta}\int_{B_{\rho^{\prime}}}u^{q}dx\leqq C,
\]
for any $\rho^{\prime}<\bar{\rho},$ from Lemma \ref{pol}. Then (ii) does not
hold.\textbf{\medskip}
\end{proof}

\begin{remark}
Under the assumption (ii) or (iii) of Proposition \ref{comp}, for any ball
$B_{\rho}=B(x_{0},\rho)\subset\Omega$, $\int_{B_{\rho}}u(.,t)$ and
$\int_{B_{\rho}}v(.,t)$ are unbounded near $0$, from Lemma \ref{doux}. But we
cannot prove that $\lim_{t\rightarrow0}\int_{B(x_{0},\rho)}u(.,t)=\infty$ or
$\lim_{t\rightarrow0}\int_{B(x_{0},\rho)}v(.,t)=\infty,$ even in the case
$p,q>1$ where (\ref{sum}) holds.\textbf{\medskip}
\end{remark}

We give a last result concerning the case where the two equations are
sublinear.\medskip

\begin{proposition}
\label{subli} Assume $0<p,q\leqq1.$ Let $(u,v)$ be a positive solution of the
system in $\Omega\times\left(  0,T\right)  .$ Then there exist two nonnegative
Radon measures $\mu_{1},\mu_{2}$ on $\Omega,$ such that for any $\psi\in
C_{c}^{0}(\Omega\mathcal{)}$
\[
\lim_{t\rightarrow0}\int_{\Omega}u(.,t)\psi=\int_{\Omega}\psi d\mu_{1}%
,\qquad\lim_{t\rightarrow0}\int_{\Omega}v(.,t)\psi=\int_{\Omega}\psi d\mu
_{2},
\]

\end{proposition}

\begin{proof}
Consider any ball $B(x_{0},\rho)\subset\Omega,$ and assume $x_{0}=0.$ Consider
again $X,Y,Z,W$ defined by (\ref{xw}). Here we find
\[
W(t)=\int_{B_{\rho}}v^{p}(.,t)\xi^{\lambda}\leqq\int_{B_{\rho}}(v(.,t)+1)\xi
^{\lambda}\leqq Y(t)+C,
\]%
\[
Z(t)=\int_{B_{\rho}}u^{q}(.,t)\xi^{\lambda}\leqq\int_{B_{\rho}}(u(.,t)+1)\xi
^{\lambda}\leqq X(t)+C,
\]
and
\[
\frac{d}{dt}(e^{\lambda\lambda_{1,\rho}t}X(t))+e^{\lambda\lambda_{1,\rho}%
t}W(t)\geqq0,\qquad\frac{d}{dt}(e^{\lambda\lambda_{1,\rho}t}Y(t))+e^{\lambda
\lambda_{1,\rho}t}Z(t)\geqq0,
\]
then the function $\Phi=e^{\lambda\lambda_{1,\rho}t}(X(t)+Y(t))$ satisfies
$\Phi^{\prime}(t)+\Phi(t)+Ce^{\lambda\lambda_{1,\rho}t}\geqq0$, that is
($e^{t}(\Phi(t)+C(1+\lambda\lambda_{1,\rho})^{-1}e^{\lambda\lambda_{1,\rho}%
t})^{\prime}\geqq0$. Then $\Phi(t)$ has a limit as $t\longrightarrow
0.$\textbf{\medskip}
\end{proof}

\textbf{Open problems: \medskip}

1) Can we extend Theorem \ref{trace} to the case $pq>1?$

2) Can we extend Proposition \ref{subli} to the case $pq<1?$

\section{Removability results\label{remove}}

Here we prove the removability of punctual singularities when $p$ and $q$ are
supercritical.\medskip

\begin{proof}
[Proof of Theorem \ref{rem}]We can assume that $q\geq p\geq1+2/N.$ Let
$\omega$ be a regular domain such that $\omega\subset\subset\Omega
\backslash\left\{  0\right\}  $ and let $T_{1}<T.$ Then from (\ref{dil})
$u,v\in L^{\infty}(0,T_{1};L^{1}(\omega));$ then from Lemma \ref{doux}, $u\in
L^{q}(\omega\times(0,T_{1}))$ and $v\in L^{p}(\omega\times(0,T_{1})).$ As in
\cite[Theorem 2]{BrFr}, step 3, the functions defined on $\omega\times(-T,T)$
by
\[
(\tilde{u},\tilde{v})(x,t)=\left\{
\begin{array}
[c]{c}%
(u,v)(x,t)\qquad\text{if }t>0,\\
0\qquad\text{if }t<0,
\end{array}
\right.
\]
satisfy $\tilde{u}\in L_{loc}^{q}(\omega\times(0,T)),$ $\tilde{v}\in
L_{loc}^{p}(\omega\times(0,T)),$ and
\begin{equation}
\tilde{u}_{t}-\Delta\tilde{u}+\tilde{v}^{p}=0,\qquad\tilde{v}_{t}-\Delta
\tilde{v}+\tilde{u}^{q}=0,\qquad\text{ in }\mathcal{D}^{\prime}\left(
\omega\times\left(  -T,T\right)  \right)  . \label{dpr}%
\end{equation}
It follows that $u,v\in C^{2,1}(\omega\times\left[  0,T\right)  )$ and
\[
u(x,0)=v(x,0)=0,\qquad\forall x\in\omega.
\]
Since $p<q,$ we have $u^{p}\leqq u^{q}+1$ from the Young inequality, thus the
function
\[
g=2^{(1-p)/p}(u+v)
\]
satisfies $g\in L^{p}(\omega\times(0,T_{1})),$ $g(x;0)=0$ on $\omega
\backslash\left\{  0\right\}  $ and
\[
g_{t}-\Delta g+g^{p}\leqq1
\]
in $\omega\times(0,T).$ Following \cite[Theorem 2]{BrFr}, step 4, we deduce
the key point estimate: there exists $C=C(N,p)$ and $\rho>0$ such that
$B(0,2\rho)\subset\Omega$, and $T_{1}<T$ such that
\begin{equation}
g(x,t)\leqq\frac{C}{(t+\left\vert x\right\vert ^{2})^{1/(p-1)}}+C,\qquad
\forall(x,t)\in B(0,\rho)\times\left(  0,T_{1}\right)  . \label{vra}%
\end{equation}
Since $p\geq1+2/N$, it implies that $g\in L^{1}(B(0,\rho)\times\left(
0,T_{1}\right)  )$. From \cite[Theorem 2]{BrFr}, step 5, it follows that $g\in
L^{p}(B(0,\rho)\times\left(  0,T_{1}\right)  ),$ thus also $u$ and $v$. We
claim that a better estimate holds, adapted to the system:
\begin{equation}%
{\displaystyle\int_{0}^{T_{1}}}
{\displaystyle\int_{B(0,\rho)}}
v^{p}<\infty\text{ \quad and }%
{\displaystyle\int_{0}^{T_{1}}}
{\displaystyle\int_{B(0,\rho)}}
u^{q}<\infty. \label{lio}%
\end{equation}
Indeed, consider a function $\zeta\in\mathcal{D}\left(  \Omega\times\left(
-T,T\right)  \right)  $ with values in $\left[  0,1\right]  $, such that
$\zeta=1$ on $B(0,\rho)\times\left(  0,T_{1}\right)  $, and a function
$\chi\in C^{\infty}(\mathbb{R})$, nondecreasing, with $\chi(t)=0$ for
$t\leqq1$, $\chi(t)=1$ for $t\geqq2;$ let $\chi_{k}(t)=\chi(kt)$ for any
$k>1.$ Setting
\[
D_{k}=\left\{  (x,t):1/k<\left\vert x\right\vert ^{2}+t<2/k\right\}  ,
\]
and using the test function
\[
\varphi_{k}(x,t)=\chi_{k}(\left\vert x\right\vert ^{2}+t)\zeta(x,t),
\]
we obtain
\begin{align}
\int_{0}^{T}\int_{B_{\rho}}u^{q}\varphi_{k}  &  =\int_{0}^{T}\int_{B_{\rho}%
}v(\varphi_{k})_{t}+\int_{0}^{T}\int_{B_{\rho}}v\Delta\varphi_{k}%
,\nonumber\label{fac}\\
\int_{0}^{T}\int_{B_{\rho}}v^{p}\varphi_{k}  &  =\int_{0}^{T}\int_{B_{\rho}%
}u(\varphi_{k})_{t}+\int_{0}^{T}\int_{B_{\rho}}u\Delta\varphi_{k}.\nonumber
\end{align}
Consequently%
\begin{equation}
\int_{0}^{T}\int_{B_{\rho}}u^{q}\varphi_{k}\leqq Ck%
{\displaystyle\iint_{D_{k}}}
v+C,\qquad\int_{0}^{T}\int_{B_{\rho}}v^{p}\varphi_{k}\leqq Ck%
{\displaystyle\iint_{D_{k}}}
u+C. \label{foc}%
\end{equation}
Next from (\ref{vra}), we have%
\begin{equation}%
{\displaystyle\iint_{D_{k}}}
(u+v)\leqq%
{\displaystyle\iint_{D_{k}}}
\left(  \frac{C}{(t+\left\vert x\right\vert ^{2})^{1/(p-1)}}+C\right)
\leqq\frac{C}{k} \label{fic}%
\end{equation}
Hence (\ref{lio}) follows from (\ref{foc}), (\ref{fic}) and the Fatou Lemma.
As a consequence of (\ref{lio}), $\tilde{u}\in L_{loc}^{q}\left(  \Omega
\times\left(  -T,T\right)  \right)  $ and $\tilde{v}\in L_{loc}^{p}\left(
\Omega\times\left(  -T,T\right)  \right)  .$

Following \cite[Theorem 2]{BrFr}, step 6, we have
\[%
{\displaystyle\iint_{D_{k}}}
(u+v)\left(  \left\vert (\chi_{k})_{t}\zeta\right\vert +\left\vert (\Delta
\chi_{k})\zeta\right\vert +\left\vert \nabla\chi_{k}\right\vert \left\vert
\nabla\zeta\right\vert \right)  \leqq C{k}%
{\displaystyle\iint_{D_{k}}}
(u+v),
\]
and from the H\"{o}lder inequality
\begin{equation}
\label{nuevo}k%
{\displaystyle\iint_{D_{k}}}
(u+v)\leqq k\left(
{\displaystyle\iint_{D_{k}}}
(u+v)^{p}\right)  ^{1/p}\left\vert D_{k}\right\vert ^{1/p^{\prime}}\leqq
C\left(
{\displaystyle\iint_{D_{k}}}
(u+v)^{p}\right)  ^{1/p}.
\end{equation}
Since the right hand side of (\ref{nuevo}) tends to 0 from (\ref{lio}), we can
pass to the limit as $k\rightarrow\infty$ in (\ref{fac}), and obtain
\[
\int_{0}^{T}\int_{B_{\rho}}u^{q}\zeta=\int_{0}^{T}\int_{B_{\rho}}v\zeta
_{t}+\int_{0}^{T}\int_{B_{\rho}}v\Delta\zeta,\qquad\int_{0}^{T}\int_{B_{\rho}%
}v^{p}\zeta=\int_{0}^{T}\int_{B_{\rho}}u\zeta_{t}+\int_{0}^{T}\int_{B_{\rho}%
}u\Delta\zeta,
\]
and then
\[
\tilde{u}_{t}-\Delta\tilde{u}+\tilde{v}^{p}=0,\qquad\tilde{v}_{t}-\Delta
\tilde{v}+\tilde{u}^{q}=0,\qquad\text{ in }\mathcal{D}^{\prime}\left(
\Omega\times\left(  -T,T\right)  \right)  .
\]
Therefore $\tilde{u},\tilde{v}\in C^{2,1}\left(  \Omega\times\left(
-T,T\right)  \right)  $, and $u(x,0)=v(x,0)=0$ on $\Omega.$\textbf{\medskip}
\end{proof}

\textbf{Open problem: }In the elliptic problem (\ref{ell}) in
$B(0,1)\backslash\left\{  0\right\}  $, it was shown in \cite[Corollary
1.2]{BiGr} that the singularities at $0$ are removable as soon as
\[
\max(2a,2b)\leqq N-2.
\]
In the case of system (\ref{one}), an open question is to know if the initial
punctual singularities at $0$ are removable whenever
\[
\max(a,b)\leqq\frac{N}{2},
\]
a condition which is obviously satisfied when $p,q\geqq1+2/N$.

\end{document}